\documentclass[12pt]{article}

\usepackage{latexsym}
\usepackage[margin=1in]{geometry}
\usepackage{float}
\usepackage{amsfonts}
\usepackage{amssymb}
\usepackage{graphicx}    
\usepackage{tikz}
\usepackage{tkz-graph}
\usepackage{color}

\usepackage{amstext,amsmath} 
\usepackage{amsthm}
\usepackage{verbatim}
\usepackage{enumerate}
\usepackage[all]{xy}

\newtheorem{theorem}{Theorem}[section]
\newtheorem{thm}[theorem]{Theorem}

\newtheorem{lem}[theorem]{Lemma}
\newtheorem{cor}[theorem]{Corollary}

\newtheorem{prop}[theorem]{Proposition}

\newcommand\cg[1]{ \mathcal{C}(#1)}

\DeclareMathOperator{\CC}{cc}

\begin{document}
\title{Compressed Cliques Graphs, Clique Coverings and Positive Zero Forcing}

\author{
  Shaun Fallat
    \thanks{Department of Mathematics and Statistics, University of Regina,
    Regina, Saskatchewan S4S 0A2, Canada. Research supported in part by an NSERC Discovery Research Grant,
    Application No.: RGPIN-2014-06036.
       Email: \texttt{Shaun.Fallat@uregina.ca}.}$\,$ \thanks{Corresponding Author}
  \and
  Karen Meagher
    \thanks{Department of Mathematics and Statistics, University of Regina,
    Regina, Saskatchewan S4S 0A2, Canada. Research supported in part by an NSERC Discovery Research Grant,
    Application No.: RGPIN-341214-2013.
       Email: \texttt{Karen.Meagher@uregina.ca}.}
  \and
  Abolghasem Soltani
    \thanks{Department of Mathematics and Statistics, University of Regina,
    Regina, Saskatchewan S4S 0A2, Canada.
    Research supported in part by an NSERC Discovery Research Grant,
    Application No.: RGPIN-2014-06036.
       Email: \texttt{abel.soltani@gmail.com}.}
  \and
  Boting Yang
    \thanks{Department of Computer Science, University of Regina,
        Regina, Saskatchewan S4S 0A2, Canada. Research supported in part by an NSERC Discovery Research Grant, Application No.: RGPIN-2013-261290.
        Email: \texttt{Boting.Yang@uregina.ca}.}
}


\maketitle

\begin{abstract}
  Zero forcing parameters, associated with graphs, have been studied
  for over a decade, and have gained popularity as the number of
  related applications grows. In particular, it is well-known that
  such parameters are related to certain vertex coverings. Continuing
  along these lines, we investigate positive zero forcing within the
  context of certain clique coverings. A key object considered here is
  the \emph{compressed cliques graph}. We study a number of properties
  associated with the compressed cliques graph, including: uniqueness,
  forbidden subgraphs, connections to Johnson graphs, and positive
  zero forcing.

 \vspace{.4cm}


\vspace{.4cm}
\end{abstract}

\section{Introduction}

Suppose that $G$ is a simple finite graph with vertex set $V=V(G)$ and edge
set $E=E(G)$. We use $\{u, v\}$ to denote an edge with endpoints $u$ and $v$.
Further, for a graph $G=(V,E)$ and $v \in V$, the vertex set $\{u: \{u,v\} \in
E\}$ is the \emph{neighbourhood} of $v$, denoted as $N_G(v)$, and the size of
the neighbourhood of $v$ is called the \emph{degree} of $v$. For $V'
\subseteq V$, the vertex set $\{x: \{x,y\} \in E, x \in V \setminus
V'$ and $y\in V' \}$ is the \emph{neighbourhood} of $V'$, denoted as
$N_G(V')$. Also, we let the set $N_{G}[v] = \{v \} \cup N_{G}(v)$
denote the \emph{closed neighbourhood} of the vertex $v$.  We use $G[V']$ to
denote the subgraph induced by $V'$, which consists of all vertices of
$V'$ and all of the edges in $G$ that contain only vertices from $V'$.  We
use $G-v$ to denote the subgraph induced by $V \setminus \{v\}$. 

For an integer $n \geq 1$, we let $K_{n}$ denote the complete graph on $n$
vertices. We will also refer to a complete graph on $n$ vertices as a \emph{clique}
on $n$ vertices. A \emph{cycle} on $n$ vertices $\{v_1, v_2, \ldots, v_n\}$ is a graph with
edges $E=\{ \{v_1, v_2\}, \{v_2, v_3\}, \ldots, \{v_{n-1}, v_n\}, \{v_n,v_1\}\}$. Such a 
cycle will also be referred to as an $n$-cycle or a cycle of length $n$. 

Our interest in this work is to consider how positive zero forcing
sets are related to cliques in a graph and, further, clique intersection,
and clique coverings. Zero forcing on a graph was originally designed to be used
as a tool to bound the maximum nullity associated with collections of symmetric matrices 
derived from a graph $G$ (see, for example, \cite{MR2388646}). Positive zero forcing was
an adaptation of conventional zero forcing to play a similar role for positive semidefinite
matrices (see \cite{zplus}).

Zero forcing in general is a graph colouring problem in which an
initial set of vertices are coloured black, while the remaining
vertices are coloured white. Using a designated colour rule, the
objective is to change the colour of as many white vertices to black
as possible. There are two common rules, these are known as zero
forcing and positive zero forcing. The process of a black vertex $u$
changing the colour of a white vertex $v$ to black is usually referred
to as ``$u$ forces $v$''. The size of the smallest initial set of
black vertices that will ``force'' all vertices black is called
either the \emph{zero forcing number} or the \emph{positive zero
  forcing number} of $G$ depending on which rule is used. This number
is denoted by either $Z(G)$ or $Z_+(G)$ (again, depending on which rule
is used).

Here we are more interested in the behaviour of the positive zero
forcing number in connection with cliques and clique coverings in a
graph. In particular, we consider $Z_{+}(G)$, when maximal cliques of
$G$ satisfy certain intersection properties.  Consequently, we now
carefully review some basic terminology associated with positive zero
forcing in a graph.

The positive zero forcing rule is also based on a colour change rule
similar to the zero forcing colour change rule (see \cite{zplus} and
also \cite{ekstrand2011positive} and \cite{ekstrand2011note}).  In
this case, suppose $G$ is a graph and $B$ a subset of vertices; we
initially colour all of the vertices in $B$ black, while all remaining
vertices are designated white. Let $W_1,\dots,W_k$ be the sets of
vertices in each of the connected components of $G$ after removing the
vertices in $B$.  If $u$ is a vertex in $B$ and $w$ is the only white
neighbour of $u$ in the graph induced by the subset of vertices
$W_i\cup B$, then $u$ can force the colour of $w$ to black. This rule
is called the \emph{positive colour change rule}.  The size of
the smallest positive zero forcing set of a graph $G$ is denoted by
$Z_+(G)$. For all graphs $G$, since a zero forcing set is also a
positive zero forcing set we have that $Z_+(G) \leq Z(G)$. A number of
facts have been demonstrated for the positive zero forcing number,
see, for example, \cite{zplus}. If a subset $S$ of $V(G)$ is a
positive zero forcing set with $|S|=Z_{+}(G)$, then we refer to $S$ as
an {\em optimal} positive zero forcing set for $G$.

It is known that by following the sequence of forces throughout the
conventional zero forcing process, a path covering of the vertices is
derived (see \cite[Proposition 2.10]{zplus} for more
details). When the positive colour change rule is applied, two or
more vertices can perform forces at the same time, and a vertex can
force multiple vertices from different components at the same time.
This implies that the positive colour change rule produces a
partitioning of the vertices into sets of vertex disjoint induced
rooted trees, which we will refer to as \emph{forcing trees}, in the graph.

Given $K_n$, the complete graph on $n \geq 2$ vertices, it is not difficult to observe that
\[ Z_{+}(K_n) = Z(K_n) = n-1.\] Furthermore, even though the parameters $Z$ and $Z_{+}$ are not
generally monotone on induced subgraphs, it is true that if $G$ contains a clique on $k$ vertices,
then both $Z(G)$ and $Z_{+}(G)$ are at least $k-1$. So in some sense, cliques in a graph play an 
important role in determining both zero forcing and positive zero forcing sets in a graph. We
explore this correspondence further in this paper. As an example, consider chordal graphs (that is,
graph with no induced cycles of length 4 or more). For chordal graphs it is known that the positive
zero forcing number is equal to the number of vertices minus the fewest number of cliques that contain
all of the edges (see, for example, \cite{FMY}).

Our paper is divided into twelve sections. The next two sections deal with
certain types of clique coverings, followed by two sections devoted to
a graph, and its uniqueness, that results from these identified clique
coverings, which we call a \emph{compressed cliques graph}. In the
sixth the seventh sections, we discuss the clique covering number of the
compressed cliques graph and bounds on the positive zero forcing
number. Section 8 connects compressed cliques graphs with certain
well-studied Johnson graphs, and Section 9 is concerned with forbidden subgraphs associated
with compressed cliques graphs. Following this, the next two sections discuss examples,
including a new family of graphs called the \emph{vertex-clique} graph
and a related concept we call the \emph{reduced graph}, along with additional examples and facts
relating the positive zero forcing number and compressed cliques graphs. We conclude with a brief 
outline for potential future research along these lines.

\section{Simply intersecting clique coverings}

Recall that a \emph{clique} in a graph is a subset of vertices which induces a
complete subgraph.  A clique in a graph is \emph{maximal} if no
vertex in the graph can be added to it to produce a larger clique.  A
\emph{clique covering} of a graph is a set of cliques with the property
that every edge is contained in at least
one of the subgraphs induced by one of the cliques in the set. (Note
that unless the graph has isolated points, a clique covering that
contains every edge also contains every vertex.) The \emph{size} of
a clique covering is the number of cliques in the covering. For a graph $G$, we
denote the size of a smallest clique covering by $\CC(G)$. Further, we call a given clique covering
\emph{minimal} if the number of cliques in this covering is equal to $\CC(G)$. Observe that a graph $G$
is a complete graph if and only if $\CC(G) = 1$.

A clique covering for a graph $G$ is called a {\em min-max clique
  covering} if its size is $\CC(G)$ and every clique in it is maximal.
Let $G$ be a graph and let $\{C_1,C_2, \dots, C_\ell\}$ be a min-max
clique covering.  Any minimal clique covering can be transformed into a min-max
clique covering by appropriately adding vertices to the non-maximal
cliques.

\begin{prop}
  For any graph $G$, there exists a clique covering with size $\CC(G)$
  in which every clique is maximal, that is, there is always a min-max clique
  covering of $G$. 
\end{prop}

We note that there are graphs (e.g. the wheel graph on at least seven
vertices) for which there exist multiple clique covers of size
$\CC(G)$ in which not all cliques are maximal. Further, a min-max
clique covering for a graph may not be unique. Consider the graph,
$\mathrm{circ}(6, \{1, 2\})$, given in
Figure~\ref{fig:circulant}. This graph has two min-max clique covers,
which are identified in Section~\ref{sec:nonunique}.

Let $\mathcal{C} = \{C_1,\dots,C_\ell\}$ be a clique covering for $G$.
If for any set of distinct triples $i,j,k \in \{1, \dots, \ell\}$ it
is the case that $C_i \cap C_j \cap C_k = \emptyset$, we say the clique
covering has {\em simple intersection}. If the clique covering has
simple intersection, then a vertex that is in $C_i \cap C_j$ (where $i
\neq j$) does not belong to any other clique in $\mathcal{C}$. 

There are many examples of graphs with a clique covering with simple
intersection. It is not hard to see that the only tree that possesses
a clique covering satisfying simple intersection is a path. For a
cycle on $n$ vertices, the set of edges forms a min-max clique
covering with simple intersection.  It turns out that both of the
min-max clique covers of $\mathrm{circ}(6, \{1, 2\})$ given in the
proof of Theorem~\ref{thm:unique} have simple intersection. The wheel
on at least seven vertices does not have a minimal clique covering
that satisfies simple intersection; the center vertex must belong to
at least three cliques in such a clique covering.  As a final example,
consider the graph in Figure~\ref{fig:2-tree}. This graph has clique
cover number equal to three and a clique cover that satisfies simple
intersection, but it has no min-max clique cover that satisfies simple
intersection. Throughout this paper, we will only consider the
property of simple intersection for min-max clique coverings.

\begin{figure*}[h]
\begin{center}
\begin{tikzpicture}[scale=.75]
\GraphInit[vstyle=simple]
\tikzset{VertexStyle/.style = {shape = circle,fill = black,minimum size = 5pt,inner sep=0pt}}
\Vertex[x=0, y=1.5] {A}
\Vertex[x=1, y=0] {B}
\Vertex[x=2, y=1.5] {C}
\Vertex[x=3, y=0] {D}
\Vertex[x=4, y=1.5] {E}
\Edge(A)(B) \Edge(B)(C)
\Edge(C)(D) \Edge(B)(D)
\Edge(A)(C) \Edge(C)(E)
\Edge(D)(E) 
\end{tikzpicture}
\end{center}
\caption{No min-max clique cover of this graph has simple intersection.\label{fig:2-tree}}
\end{figure*}
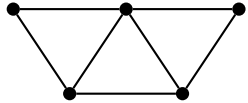

\section{Non-unique clique coverings}
\label{sec:nonunique}

Let $G$ be a graph with $n$ vertices, which are labeled $0, 1, \dots,
n-1$, and let $S=\{k_1, k_2, \dots, k_\ell\}$ be a set of positive
integers such that $k_1< k_2 < \dots < k_\ell < (n+1)/2$. 
The \emph{circulant graph}, denoted by $\mathrm{circ}(n,S)$, has each
vertex $i$ in $\{0, \dots, n-1\}$ adjacent to $i \pm k_1$, $i \pm k_2$, $ \dots,
i \pm k_\ell \pmod{n}$. The graph $\mathrm{circ}(6, \{1, 2\})$ plays a key
role in identifying graphs that possess a unique min-max clique
covering satisfying simple intersection. This graph is isomorphic to
the graph in Figure~\ref{fig:circulant}. This graph has one vertex for
each subset of $\{1,2,3,4\}$ of size two, and two vertices are
adjacent if the sets intersect. This graph is also known as the
Johnson graph $J(4,2)$ (see also Section 8).

The graph $\mathrm{circ}(6,\{1,2\})$ (see Figure \ref{fig:circulant}) has two min-max clique
coverings. One clique covering of this graph is $\{C_1,C_2,C_3,C_4\}$
where $C_i$ is the set of all vertices with a label that contains an
$i$.  A second clique covering is formed by the following sets of
vertices of $\mathrm{circ}(6,\{1,2\})$:
\begin{align*}
& \{ \{1,2\}, \{1,3\}, \{2,3\} \},\quad  \{ \{1,2\}, \{1,4\}, \{2,4\} \}, \\
& \{ \{1,3\}, \{1,4\}, \{3,4\} \} ,\quad  \{ \{2,3\}, \{2,4\}, \{3,4\} \}.
\end{align*}

\begin{figure*}[h]
\begin{center}
\begin{tikzpicture}[scale=.75]
\tikzset{VertexStyle/.style = {shape = circle,fill = black,minimum size = 5pt,inner sep=0pt}}
\Vertex[L={$\{3,4\}$}, x=1, y=0,LabelOut=true,Lpos=270] {A}   
\Vertex[L={$\{2,4\}$}, x=-0.5, y=1, LabelOut=true,Lpos=270] {B}    
\Vertex[L={$\{2,3\}$}, x=2.5, y=1, LabelOut=true,Lpos=0] {C}
\Vertex[L={$\{1,2\}$}, x=1, y=3.5, LabelOut=true,Lpos=90] {F}
\Vertex[L={$\{1,4\}$}, x=-0.5, y=2.5, LabelOut=true,Lpos=180] {D} 
\Vertex[L={$\{1,3\}$}, x=2.5, y=2.5, LabelOut=true,Lpos=0] {E}
\Edge(A)(B) \Edge(B)(C)
\Edge(C)(F) \Edge(B)(D)
\Edge(A)(C) \Edge(C)(E)
\Edge(D)(E) \Edge(D)(F)
\Edge(E)(F) \Edge(B)(F)
\Edge(A)(D) \Edge(A)(E)
\end{tikzpicture}
\end{center}
\caption{$\mathrm{circ}(6, \{1,2\})$.\label{fig:circulant}}
\end{figure*}

The next fact is a key to characterizing the graphs that possess more than one min-max clique covering.

\begin{theorem}\label{thm:unique}
  If a graph $G$ has two distinct min-max clique covers that both
  satisfy simple intersection, then $G$ contains $\mathrm{circ}(6,\{1,2\})$ as
  an induced subgraph.
\end{theorem}
\begin{proof}
  Let $\mathcal{C} =\{C_1,C_2,\dots,C_k\}$ be a min-max clique
  covering for a graph $G$ with simple intersection. Label the
  vertices in $G$ by the cliques that contain them: if vertex $v$ is
  in both $C_i$ and $C_j$, then $v$ gets the label $\{i,j\}$; if $v$
  is in $C_i$, but no other cliques, then $v$ gets the label
  $\{i\}$. Distinct vertices may have the same label. Since
  $\mathcal{C}$ has simple intersection, every vertex has a label with
  size $1$ or $2$. Since every edge must be in at least one clique of
  $\mathcal{C}$, two vertices are adjacent in $G$ if and only if their
  labels are intersecting sets. Further, the labels on the vertices in
  the clique $C_j$ must be in the set $\{\{j\},\{j,1\},\{j,2\},\dots,
  \{j,k\}\}$.

  Assume that $\mathcal{D} = \{D_1, D_2, \dots, D_k\}$ is another
  min-max clique covering of $G$. If there is a vertex in $G$ with the
  label $\{i\}$, then there is a clique in $\mathcal{D}$ that contains
  the vertex. Further, every other vertex in this clique must have a label
  that contains $i$. Since $\mathcal{D}$ is a min-max clique covering,
  this clique must be maximal and thus must be exactly the clique
  $C_i$, so we know that in this case, $C_i \in \mathcal{D}$. Further,
  if the labels of all the vertices in some $D \in \mathcal{D}$ have a common
  element, say $i$, then $D =C_i$ (again since $D$ is maximal).

  Since $\mathcal{C}$ and $\mathcal{D}$ are not equal, we can assume
  that $C_1 \not \in \mathcal{D}$. This implies that $C_1$ does not
  contain any vertices labeled with $\{1\}$, and $C_1$ must contain
  at least two vertices with distinct labels. Assume that $C_1$
  contains vertices labeled with $\{1,2\}$ and $\{1,3\}$.  The edge
  between vertices labeled $\{1,2\}$ and $\{1,3\}$ must be in a
  clique of $\mathcal{D}$, say $D_a$. The labels on the vertices in
  $D_a$ must intersect, but not all contain $1$ as a common
  element. Thus we can assume without loss of generality that there
  are three vertices in $D_a$ that have the labels 
  $\{\{1,2\},\{1,3\},\{2,3\}\}$.

  We know that the clique $C_1$ contains vertices labeled $\{1,2\}$
  and $\{1,3\}$. Further, there must be a vertex in $C_1$ with the
  label $\{1,x\}$, where $x \not \in \{1,2,3\}$, since if there was no
  such vertex, then any vertex labeled with $\{2,3\}$ could be added
  to $C_1$, this contradicts the fact that $C_1$ is maximal (and not in
  $\mathcal{D}$).

  Similarly, the clique $C_2 \in \mathcal{C}$ contains vertices
  labeled with $\{1,2\}$ and $\{2,3\}$. Since $C_2$ must be maximal,
  it contains another vertex (otherwise we could add the vertex
  labeled $\{1,3\}$ to it). Assume that this extra vertex is labeled
  by $\{2,y\}$ (here we could have that $y=2$) with $y \neq 1,3$.

  Consider the vertex labeled $\{1,2\}$ in the clique covering
  $\mathcal{D}$. There is an edge between it and the vertex labeled
  $\{1,x\}$ and the vertex labeled $\{2,y\}$. Since the vertex
  labeled $\{1,2\}$ is already in the clique $D_a$ and $\mathcal{D}$
  has simple intersection, these two edges must be in the same clique
  of $\mathcal{D}$, which implies that $x=y$. Since $x \neq 2,3$, we
  will assume without loss of generality that $x=y = 4$. (Note that
  $\{1,x\}$ cannot be in the clique $D_a$, since $x \not \in
  \{1,2,3\}$ and $\{2,y\} \not \in D_a$ since $y \neq 1,3$.) Thus
  there are vertices labeled with $\{1,4\}$ and $\{2,4\}$.

  Next consider the clique $C_3$, this clique contains vertices
  labeled by $\{1,3\}$ and $\{2,3\}$. Since $C_3$ is maximal and it
  does not contain the vertex labeled $\{1,2\}$, it must contain a
  vertex labeled $\{3,z\}$, where $z \neq 1,2$ (but it could be $3$).

  Consider the vertex labeled $\{1,3\}$, this vertex is adjacent to
  the vertex labeled $\{1,4\}$ and the vertex labeled $\{3,z\}$. In
  the clique covering $\mathcal{D}$, the vertex labeled $\{1,3\}$ is
  in $D_a$. But neither $\{1,4\}$, nor $\{3,z\}$ are in $D_a$
  (as $\{1,4\}$ does not intersect with $\{2,3\}$, and $\{3,y\}$ does not
  intersect with $\{1,2\}$). Since $\mathcal{D}$ has simple
  intersection, $\{1,3\}$ can only be in one other clique. So the
  vertices labeled with $\{1,4\}$ and $\{3,z\}$ must in this other
  clique in $\mathcal{D}$. This means that they are adjacent and that
  $z=4$. Thus there is a vertex labeled by $\{3,4\}$.

  It now follows that $G$ contains a subgraph isomorphic to
  $\mathrm{circ}(6,\{1,2\})$ (see Figure~\ref{fig:circulant}).
\end{proof}

Unfortunately,  the converse to Theorem~\ref{thm:unique} is false in
general. A simple example can be derived from the graph $\mathrm{circ}(6,
\{1, 2\})$ by adding an additional vertex, see
Figure~\ref{fig:circulant-1}.  It is easy to verify that this graph
has a unique min-max clique covering that satisfies simple
intersection, but certainly contains $\mathrm{circ}(6,\{1,2\})$ as an induced
subgraph.

\begin{figure*}[h]
\begin{center}
\begin{tikzpicture}[scale=.75]
\tikzset{VertexStyle/.style = {shape = circle,fill = black,minimum size = 5pt,inner sep=0pt}}
\Vertex[L={$\{3,4\}$}, x=1, y=0,LabelOut=true,Lpos=270] {A}   
\Vertex[L={$\{2,4\}$}, x=-0.5, y=1, LabelOut=true,Lpos=180] {B}    
\Vertex[L={$\{2,3\}$}, x=2.5, y=1, LabelOut=true,Lpos=0] {C}
\Vertex[L={$\{1,2\}$}, x=1, y=3.5, LabelOut=true,Lpos=90] {F}
\Vertex[L={$\{1,4\}$}, x=-0.5, y=2.5, LabelOut=true,Lpos=180] {D} 
\Vertex[L={$\{1,3\}$}, x=2.5, y=2.5, LabelOut=true,Lpos=0] {E}
\Vertex[L={$\{4\}$}, x=-2, y=0, LabelOut=true,Lpos=2700] {G}

\Edge(A)(B) \Edge(B)(C)
\Edge(C)(F) \Edge(B)(D)
\Edge(A)(C) \Edge(C)(E)
\Edge(D)(E) \Edge(D)(F)
\Edge(E)(F) \Edge(B)(F)
\Edge(A)(D) \Edge(A)(E)
\Edge(G)(A) \Edge(G)(B) \Edge(G)(D)
\end{tikzpicture}
\end{center}
\caption{A graph that contains $\mathrm{circ}(6, \{1,2\})$, and has a unique min-max clique cover.\label{fig:circulant-1}}
\end{figure*}
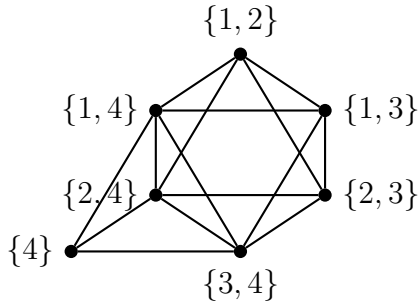

\section{Compressed cliques graphs}

If a graph has a min-max clique covering that satisfies simple intersection,
then the graph can be simplified in a way that allows us to determine
both the clique covering number and the positive zero forcing number of the
original graph from the simplified graph.

Let $G$ be a graph and $S=\{v_1, v_2, \dots, v_k\}$ a set of vertices
in $G$. The {\em contraction of $S$ in $G$} is the graph formed by
replacing the vertices in $S$ by a single vertex $v_S$, where $v_S$ is
adjacent to any vertex in $V(G) \setminus S$ that is adjacent to a
vertex in $S$.

Let $G$ be a graph with a min-max clique covering $\mathcal{C} =
\{C_1,C_2, \dots, C_{\ell}\}$ that has simple intersection.  We
can construct a new graph that is related to $G$ called a {\em
  compressed cliques graph}. For distinct $i,j \in\{1, \dots, \ell\}$
define the sets
\[
C_{i,j} = C_i \cap C_j, \quad  
C_{i,i} = C_i \backslash \bigcup_{\stackrel{j \in  \{1,\dots, \ell\}}{j \neq i}}C_j,
\]
(these sets may be empty). First, for each pair of distinct $i,j \in \{1,
\dots, \ell\}$, if $C_{i,j}$ is non-empty, then contract all the
vertices in $C_{i,j}$ to a single vertex labeled $v_{i,j}$. Second, if
$C_{i,i}$ is non-empty, then contract all the vertices in $C_{i,i}$ to
a single vertex; label this vertex $v_{i,i}$. In this graph the
vertices $v_{i,j}$ and $v_{i',j'}$ are adjacent if and only if the
sets $\{i,j\}$ and $\{i',j'\}$ have non-empty intersection.

We have seen that a graph may have multiple min-max clique covers that
satisfy simple intersection. For example the graph
$\mathrm{circ}(6,\{1,2\})$. However, for this graph, the compressed cliques
graphs, for either clique cover are isomorphic.  In fact, both are isomorphic to 
$\mathrm{circ}(6, \{1,2\})$.
Thus, an obvious
question is to verify that the associated compressed cliques graphs
are isomorphic in the presence of distinct min-max cliques covers satisfying
simple intersection (see the next section). Assuming this fact, we denote {\em the} compressed
cliques graph of $G$ by $\cg{G}$ and give some simple examples.  For
any integer $n$,
\[
\cg{K_n} = K_1, \qquad \cg{K_n\backslash \{e\}} = P_3 \, (n \geq 3), \qquad 
\cg{C_n}= C_n, \qquad \cg{P_n} = P_{n},
\]
(where $K_n\backslash \{e\}$ is a complete graph with an edge removed,
and $P_n$ is the path with $n$ vertices).

Looking at the last two examples, a natural question to ask is when is
the compressed cliques graph $\cg{G}$ isomorphic to $G$?

We will make use of the following map $\phi: V(G) \longrightarrow V(\cg{G})$ defined as
\begin{align}\label{phidefined}
\phi(v) =
\begin{cases}
v_{i,j} & \textrm{ if } v \in C_i \cap C_j, \\
v_{i,i} & \textrm{ if $v \in C_i$ and no other cliques.}
\end{cases}
\end{align}

\begin{thm} \label{thm:C(G)isG}
  Let $G$ be a graph with $n$ vertices and $\mathcal{C}=\{C_1,C_2,
  \ldots, C_k \}$ be a min-max clique covering of $G$. If
  $\mathcal{C}$ satisfies simple intersection, then $\mathcal{C}(G)$
  isomorphic to $G$ if and only if all of the sets $C_{i,i}$ and
  $C_{i,j}$ for $i,j \in \{1,\dots, k\}$ contain no more than one
  vertex.
\end{thm}

\begin{proof}
For every $v \in V(G)$, there is a unique pair $(i,j) \in
\{1,\dots,n\} \times \{1,\dots,n\}$ such that $v \in C_{i,j}$ (it
is possible that $i=j$).

Consider the map $\phi$ from $V(G)$ to $V(\mathcal{C}(G))$ defined
above. If all of the sets $C_{i,i}$ and $C_{i,j}$ for $i,j \in
\{1,\dots, k\}$ contain no more than one vertex this mapping $\phi$ is
an edge, and non-edge, preserving bijection.  Hence $\mathcal{C}(G)
\cong G$.

Conversely, the graph $\mathcal{C}(G)$ has one vertex for each
non-empty set $C_{i,j}$. These sets cover the vertices in $G$, so if
one of these sets contains two or more vertices, then $\mathcal{C}(G)$
will have less vertices than $G$, and hence these two graphs cannot be
isomorphic.
\end{proof}

\begin{cor} \label{ccG} 
  Suppose $G$ is a graph that possesses a min-max clique covering having
  simple intersection. Then $\mathcal{C}(\mathcal{C}(G))=
  \mathcal{C}(G)$.
\end{cor}

The map $\phi$ will also be used in Theorem~\ref{ccofCG}, there we
will need the following fact.

\begin{prop}\label{prop:phiconnects}
  If there is a path from $u$ to $v$ in $G$, then either $\phi(u) =
  \phi(v)$ or there is a path from $\phi(u)$ to $\phi(v)$ in $\cg{G}$.
\end{prop}

\section{Uniqueness of compressed cliques graphs}

We have seen that the circulant graph $\mathrm{circ}(6, \{1, 2\})$ has
two distinct min-max clique covers that both satisfy simple
intersection---see Figure~\ref{fig:circulant}.  Furthermore, if we
expand each vertex of $\mathrm{circ}(6, \{1, 2\})$ to a clique of any
positive size and expand edges correspondingly (join the vertices of
these cliques if the corresponding vertices were adjacent in
$\mathrm{circ}(6, \{1, 2\})$), all of the resulting graphs have two
distinct min-max clique covers that both satisfy simple
intersection. In order to ensure the concept of a compressed cliques
graph well-defined, we need to consider the uniqueness of the
compressed cliques graph up to isomorphism. In this section we will
show that if a graph has multiple min-max clique covers that satisfy
simple intersection, the corresponding compressed cliques graphs are
unique, up to isomorphism.

\begin{lem} \label{lem:neighbor} For a graph $G$, let $\mathcal{C}$ be
  a min-max clique cover of $G$ that satisfies simple
  intersection. For a vertex $v \in V(G)$ and two distinct cliques $C_1,
  C_2 \in \mathcal{C}$, we have $v \in C_1 \cap C_2$ if and only if $N_{G}[v]
  = C_1 \cup C_2$.
\end{lem}

\begin{proof} 
  Suppose $v \in C_1 \cap C_2$. Since $\mathcal{C}$ satisfies simple
  intersection, all neighbours of $v$ must be in $C_1$ or $C_2$. Thus
  $N_G[v] \subseteq C_1 \cup C_2$. For any vertex $u$ that is in $C_1$
  or $C_2$, since $v \in C_1\cap C_2$, we know that
  there is an edge between $v$ and $u$.

  On the other hand, suppose $N_G[v] = C_1 \cup C_2$. Then $v \in C_1 \cup
  C_2$. Suppose $v \in C_1 \setminus C_2$. Since $v$ is adjacent to all
  vertices of $C_2$, we know $\{v\} \cup C_2$ is also a clique, which
  contradicts the maximality of $C_2$. Hence $v \in C_1 \cap C_2$.
\end{proof}

\begin{lem} \label{lem:distinct} For a graph $G$, let $\mathcal{A}$
  and $\mathcal{B}$ be two distinct min-max clique covers of $G$ that
  both satisfy simple intersection.  Then for any vertex $v \in V(G)$,
  exactly one of the following two possibilities occurs.
\begin{enumerate}
\item If only one clique $A$ in $\mathcal{A}$ contains the
  vertex $v$, then there is a clique $B \in \mathcal{B}$ such that
  $A=B$.
\item If two distinct cliques $A, A' \in \mathcal{A}$ contain the
vertex $v$, then there are two distinct cliques $B, B' \in
\mathcal{B}$ such that $v \in B \cap B'$.
\end{enumerate}
\end{lem}
\begin{proof} Let $v$ be a vertex of $G$. Since $\mathcal{A}$ is a
  min-max clique cover of $G$ that satisfies simple intersection, we
  have the following two cases.
\begin{enumerate}
  \item  The vertex $v$ is contained in exactly one clique $A$ of
  $\mathcal{A}$. There is a clique $B \in \mathcal{B}$ such that $v
  \in B$. Suppose that $A \neq B$. Since $A$ and $B$ are maximal
  cliques, there is a vertex $w \in B \setminus A$. Since $B$ is a
  clique, $\{v, w\}$ must be an edge in $G$. Since $w \not\in A$, the
  edge $\{v, w\}$ must be covered by a clique in $\mathcal{A}
  \setminus \{A\}$. Thus, $v$ is in two cliques of $\mathcal{A}$,
  which is a contradiction. Hence $A = B$.

  \item The vertex $v$ is contained in exactly two distinct cliques $A$
  and $A'$ of $\mathcal{A}$.  If $v$ is contained in exactly one
  clique $B$ of $\mathcal{B}$, then $A$ is a proper subset of $B$,
  which contradicts the maximality of $A$. So $v$ is contained in
  exactly two cliques $B$ and $B'$ of $\mathcal{B}$, i.e., $v \in B
  \cap B'$.
\end{enumerate}
\end{proof}

\begin{lem} \label{lem:intersection} 
  For a graph $G$, let $\mathcal{A}$ and $\mathcal{B}$ be two distinct
  min-max clique covers of $G$ that both satisfy simple
  intersection. If there are two distinct cliques $A$ and $A'$ in
  $\mathcal{A}$ and two distinct cliques $B$ and $B'$ in $\mathcal{B}$
  such that $A \cap A' \cap B \cap B' \neq \emptyset$, then $A \cap
  A' = B \cap B'$.
\end{lem}
\begin{proof}
  If one of $A$ or $A'$ is equal to one of $B$ or $B'$, then it is easy to
  see that $\{A, A'\} = \{B, B'\}$ and that the result is trivial. Thus we
  only consider the case that $\{A, A'\} \cap \{B, B'\} = \emptyset$.
  Suppose $A \cap A' \neq B \cap B'$. Without loss of generality, let
  $x \in (A \cap A') \setminus (B \cap B')$.  Let $v \in A \cap A'
  \cap B \cap B' $. It follows from Lemma~\ref{lem:neighbor} that $A
  \cup A' = N[v] = B \cup B'$. So $x$ must be in $B$ or $B'$ but
  cannot be in both.

  Since $x \in A \cap A'$, from Lemma~\ref{lem:neighbor} we know that
  $N_G [x] = A \cup A'$. Thus $N_G [x] = B \cup B'$. Furthermore, it follows
  from Lemma~\ref{lem:neighbor} that $x \in B \cap B'$. This is a
  contradiction. Therefore $A \cap A' = B \cap B'$.
\end{proof}

Now we can verify that the compressed cliques graph for a graph with a
min-max clique covering with simple intersection is well-defined.

\begin{thm} \label{thm-unique}
  If a graph $G$ has a min-max clique cover that satisfies simple
  intersection, then the compressed cliques graph of $G$ is unique.
\end{thm}

\begin{proof}
  If $G$ has a unique min-max clique cover satisfying simple
  intersection, it is easy to see that the compressed cliques graph of
  $G$ is unique. Suppose that $\mathcal{A} = \{A_1,A_2, \dots,
  A_{\ell}\}$ and $\mathcal{B} = \{B_1,B_2, \dots, B_{\ell}\}$ are two
  distinct min-max clique covers of $G$ that satisfy simple
  intersection.

For distinct $i, j \in \{1, \dots, \ell\}$ define the sets
\[
A_{i,j} = A_i \cap A_j, \quad  A_{i,i} = A_i \setminus 
\bigcup_{\stackrel{j \in \{1,\dots, \ell\}}{j \neq i}}A_j ,
\]
\[
B_{i,j} = B_i \cap B_j, \quad  B_{i,i} = B_i \setminus 
\bigcup_{\stackrel{j \in \{1,\dots, \ell\}}{j \neq i}}B_j .
\]

Let $V_A$ be the set of all non-empty sets $A_{i,j}$ with $i, j \in \{1,
\dots, \ell\}$, and $V_B$ be the set of all non-empty sets $B_{i,j}$ with
$i, j \in \{1, \dots, \ell\}$.  For any $A_{i,i} \in V_A$, let $u$ be
a vertex in $A_{i,i}$. From the definition, $A_{i,i}$ is the only
clique in $\mathcal{A}$ that contains $u$. Thus it follows from
Lemma~\ref{lem:distinct}(1) that $A_{i,i} \in V_B$.  For any $A_{i,j}
\in V_A$ ($i \neq j$), let $v$ be a vertex in $A_{i,j}$. From the
definition, $A_i$ and $A_j$ are two distinct cliques in $\mathcal{A}$
that contain $v$. Thus it follows from Lemma~\ref{lem:distinct}(2)
that there are two distinct cliques $B_{i'}, B_{j'} \in \mathcal{B}$
such that $v \in B_{i'} \cap B_{j'}$. Using
Lemma~\ref{lem:intersection} we now have $A_{i} \cap A_{j} = B_{i'}
\cap B_{j'}$. Hence $A_{i,j} = B_{i',j'} \in V_B$. Therefore, $V_A
\subseteq V_B$. Similarly, we can show that $V_B \subseteq V_A$. So we
have $V_A=V_B$.

Let $E_A = \{\{A_{i,j}, A_{j,k}\}: A_i \cap A_j \not= \emptyset, A_j \cap A_k \not= \emptyset\}$ 
in which $i, j, k$ are distinct indices, and let
$E_B = \{\{B_{i,j}, B_{j,k}\}: B_i \cap B_j \not= \emptyset, B_j \cap B_k \not= \emptyset\}$ in
which $i, j, k$ are distinct indices. For any $\{A_{i,j}, A_{j,k}\} \in E_A$, there is an edge 
$\{u, v\} \in E(G)$ such that $u \in A_{i,j} = B_{i',j'}$, and $v \in A_{j,k} = B_{j'',k'}$. 
Since the edge $\{u, v\}$ is covered by a clique in ${\cal B}$, from the simple intersection 
property for the clique cover $\mathcal{B}$, we know that one of $\{B_{i'}, B_{j'}\}$ is 
the same as one of $\{B_{j''}, B_{k'}\}$. 
Thus $\{B_{i',j'}, B_{j'',k'}\} \in E_B$. So $E_A \subseteq E_B$. Similarly, we can show 
that $E_B \subseteq E_A$, and hence $E_A = E_B$. 

From the definition of the compressed cliques graph, the compressed
cliques graph of $G$ with respect to $\mathcal{A}$ is isomorphic to the
compressed cliques graph of $G$ with respect to $\mathcal{B}$.
Therefore, the compressed cliques graph of $G$ is unique.
\end{proof}

\begin{cor}
Suppose $G$ has two distinct min-max clique covers that both satisfy simple intersection.
Then $\mathcal{C}(G)=\mathrm{circ}(6, \{1, 2\})$.
\label{2covH}
\end{cor}

\begin{proof}
  Since $G$ possesses two min-max clique covers that both satisfy
  simple intersection, we know that $\mathcal{C}(G)$ contains
  $\mathrm{circ}(6, \{1, 2\})$ as an induced subgraph (Theorem
  \ref{thm:unique}).  Assume, for the sake of a contradiction, that
  $\mathcal{C}(G)\neq \mathrm{circ}(6, \{1, 2\})$.  Then there exists
  a vertex $v$ in $\mathcal{C}(G)$ adjacent to some vertex in the
  subgraph $\mathrm{circ}(6, \{1, 2\})$.  Referring to Figure
  \ref{fig:circulant}, suppose $v$ is adjacent to the vertex labeled
  $\{1,2\}$.  (Since $\mathrm{circ}(6, \{1, 2\})$ is
  vertex-transitive, the label $\{1,2\}$ is chosen without loss of
  generality.)  From the hypothesis on $G$ we know that $v$ must also
  be adjacent to any vertex whose labels contain a 1 or a 2;
  otherwise $G$ does not possess two such clique covers. In this case
  $\mathcal{C}(G)$ is not a compressed cliques graph as the vertex $v$
  should also have the label $\{1,2\}$ and hence did not exist in the
  first place. By contradiction we conclude that
  $\mathcal{C}(G)=\mathrm{circ}(6, \{1, 2\})$.
\end{proof}

Let $\mathcal{G}$ be the set of all graphs that have a min-max clique
cover satisfying simple intersection. Let $\mathcal{R}$ be a binary
relation on $\mathcal{G}$. We say two graphs $G_1, G_2 \in
\mathcal{G}$ have the relation $\mathcal{R}$ if $\mathcal{C}(G_1)=
\mathcal{C}(G_2)$. It is easy to see that $\mathcal{R}$ is an
equivalence relation, and hence induces a partition of the set $\mathcal{G}$. For any equivalence class
that contains a graph $H$, it is easy to see that the graph
$\mathcal{C}(H)$ is the minimum element in the class. Thus, the
compressed cliques graph of $\mathcal{C}(H)$ is $\mathcal{C}(H)$
itself, see also Corollary \ref{ccG}.

Before we close this section, we note that we could have used the structures  derived in the proof of Theorem \ref{thm-unique}
to deduce Corollary \ref{2covH} directly, without relying on Theorem \ref{thm:unique}. 
Since we chose to consider non-unique min-max clique covers first, we used 
Theorem \ref{thm:unique} to establish Corollary \ref{2covH}. Observe that if the compressed cliques graph of $G$  is $H$, then 
$H$ must be an induced subgraph of $G$. Hence if $G$ contains two distinct min-max clique covers that both satisfy simple intersection, we know from Corollary \ref{2covH} that $\mathcal{C}(G)=\mathrm{circ}(6, \{1, 2\})$, and hence $G$ contains $\mathrm{circ}(6, \{1, 2\})$ as an induced 
subgraph, which implies Theorem \ref{thm:unique}.

\section{Clique coverings of compressed cliques graphs}

Before we begin our analysis on positive zero forcing in compressed
cliques graphs, we consider clique coverings of compressed cliques
graphs.  We begin with the following useful lemma.  Let $\phi: V(G)
\rightarrow V(\cg{G})$ be the map defined in (\ref{phidefined}).

\begin{lem}
Let $G$ be a graph with a min-max clique covering $\mathcal{C}$ that has simple intersection.
  If $C$ is a clique in $\cg{G}$, then the preimage of $C$ under
  $\phi$ is a clique in $G$.
\end{lem} \label{func}
\begin{proof}
  Assume that $\mathcal{C}=\{C_1, \dots , C_\ell\}$ is a min-max clique
  covering of $G$ that has simple intersection.

  There are two types of cliques in $\cg{G}$. The first, are the
  cliques that are comprised of vertices that correspond to the sets
  that all contain a common element, say $i$. Clearly the preimage of
  all such vertices is contained in the clique $C_i$.

  The second type are the cliques in which the sets do not
  all contain a fixed element. These cliques must contain exactly three
  vertices of the form $\{ \{a,b\}, \{a,c\}, \{b,c\}\}$.  Let $v$ and
  $w$ be any two vertices in the preimage of these three
  vertices. Then $v$ and $w$ must each be in at least two of the
  cliques $C_a, C_b, C_c$. Thus $v$ and $w$ are both in at least one
  of these cliques so they are adjacent.
\end{proof}

\begin{thm} \label{ccofCG}
  Let $G$ be a graph in which there is a min-max clique
  covering with simple intersection. Let $\cg{G}$ be the compressed
  cliques graph of $G$.  Then
\[
\CC(G) = \CC(\cg{G}).
\]
\end{thm}
\begin{proof}
Assume that the clique cover number of $G$ is $\ell$ and that
$\{C_1,C_2,\dots, C_\ell\}$ is a min-max clique covering with 
simple intersection.

Let $D_i$ be the set of vertices in $\cg{G}$ that are labeled
$v_{i,j}$ from some $j \in \{1,2,\dots,\ell\}$ (so $j$ could equal
$i$). Then $D_i$ is a clique in $\cg{G}$. We claim that
$\{D_1,D_2,\dots,D_\ell\}$ is a clique cover for $\cg{G}$. The only
edges in $\cg{G}$ are between vertices $v_{i,j}$ and $v_{i, k}$ for
some $i,j,k \in \{1, \dots, \ell\}$, so the sets $D_i$ with $i \in
\{1,\dots,\ell\}$ cover all of the edges of $\cg{G}$. Hence the sets
$D_i$ form a clique cover of $\cg{G}$.  Thus $\CC(G) \geq
\CC(\cg{G})$.

On the other hand, let $\mathcal{C} = \{C_1, C_2, \dots, C_\ell\}$ be the
clique covering of $G$ used to define the compressed cliques graph,
and let $\mathcal{D} = \{D_1, D_2, \dots, D_\ell \}$ be any minimal clique
covering for $\mathcal{C}(G)$. Let $\phi$ be the map defined in (\ref{phidefined}). 
We claim that $\{\phi^{-1}(D_1),
\phi^{-1}(D_2), \dots, \phi^{-1}(D_\ell) \}$ is a clique covering for
$G$.  

First, if $x,y \in \phi^{-1}(D_a)$, then $\phi(x)$ and $\phi(y)$ are
adjacent in $\mathcal{C}(G)$. This implies that $x$ and $y$ are both
contained in some clique in $G$. Hence $x$ and $y$ are adjacent in $G$,
and $\phi^{-1}(D_a)$ is a clique.  

Next, let $e = \{x,y\}$ be any edge in $G$. This edge is covered by a
clique in $\mathcal{C}$, say $C_1$. Then $\phi(x) = v_{1,a}$ and
$\phi(y) = v_{1.b}$. Thus there is an edge between $\phi(x)$ and
$\phi(y)$ in $\mathcal{C}(G)$. This edge is contained in some clique
in $\mathcal{D}$, say $D_a$. Then both $x$ and $y$ are in
$\phi^{-1}(D_a)$, so the edge $\{x,y\}$ is covered.  Thus,
$\{\phi^{-1}(D_1), \phi^{-1}(D_2), \dots, \phi^{-1}(D_\ell) \}$ is a
clique covering of $G$ and $\CC(G) \leq \CC(\mathcal{C}(G))$.
\end{proof}

We now have the following interesting consequence.

\begin{cor}
  Assume that $G$ is a graph with a min-max clique covering with
  simple intersection. Let $\{D_1, \dots , D_\ell\}$ be the set of
  cliques in $\cg{G}$ defined such that $D_i$ be the set of vertices in $\cg{G}$ that are labeled
$v_{i,j}$ from some $j \in \{1,2,\dots,\ell\}$. Then the following statements hold:
\begin{enumerate}
\item The set $\{D_1, D_2, \dots, D_\ell\}$ forms a min-max clique
  cover of $\cg{G}$; and 
\item this clique cover of $\cg{G}$ has simple intersection.
\end{enumerate}
\end{cor}

\begin{proof}
  From Theorem~\ref{ccofCG}, the set $\{D_i\}$ forms a minimal clique
  cover of $\cg{G}$. So to prove the first statement, we need to show
  that each clique $D_i$ is maximal. Assume that there is a vertex $u$
  in $\cg{G}$ that is adjacent to every vertex in $D_i$, but is not in
  $D_i$.
If $v_{i,i} \in D_i$ then, since $u$ is adjacent to $v_{i,i}$, it
  must be that $u= v_{i,j}$ for some $j$; but in this case the vertex
  $u$ would be in $D_i$. 
If $D_i$ contains the single vertex $v_{i,j}$ (where $i \neq j$),
  then $C_i \subseteq C_j$, which is a contradiction.  Thus we may
  assume that $D_i$ contains the vertices $v_{i,a}$, $v_{i,b}$. Since
  $u \not \in D_i$, this implies that $u = v_{a,b}$, so $D_i$ cannot
  contain any vertices other than these two vertices.  In this case,
  $C_i = C_{i,a} \cup C_{i,b}$, and every vertex in $C_a \cap C_b$
  (since $v_{a,b} \in \cg{G}$, this intersection is non-empty) is
  adjacent to every vertex in $C_i$. This contradicts $C_i$
  being a maximal clique.

  The second statement follows easily, since if a vertex with label
  $v_{i,j}$ from $\cg{G}$ is contained in three maximal cliques, this
  contradicts the fact that the original min-max clique covering of
  $G$ satisfied simple intersection.
\end{proof}

\section{Positive zero forcing sets in compressed cliques graphs}

Using the strong connections to clique coverings, 
we now explore positive
zero forcing sets and the positive zero forcing number of 
compressed cliques graphs. This is one of our main motivations for
developing this derived graph. Our results in this section are related
to the following simple observation about positive zero forcing sets. First note
that, by definition, if $N_G[u] = N_G[v]$, then $u$ and $v$ must be adjacent.

\begin{lem}
Let $G$ be a graph and $u$ and $v$ be distinct vertices in $G$. If
$N_G[u] = N_G[v]$, then any positive zero forcing set for $G$ will contain at least
one of $u$ and $v$.
\end{lem}
\begin{proof}
Assume that $u$ and $v$ are both initially white in a colouring of $G$.
Since $u$ and $v$ are adjacent, they will always be in the same
component of $G$ after any set of vertices are removed.
Any vertex that is adjacent to one of $u$ or $v$, is adjacent to the
other, thus no vertex will ever be able to force either of them.
\end{proof}

This can be generalized to a subset of vertices as in the following fact.

\begin{cor}
  Let $G$ be a graph and $\{u_1,u_2, \dots ,u_k\}$ be a subset of $k$ vertices from
  $G$. If $N_G[u_1] = N_G[u_2]= \dots = N_G[u_k]$, then any positive
  zero forcing set for $G$ will contain at least $k-1$ of the vertices
  $\{u_1,u_2, \dots ,u_k\}$.
\end{cor}

We can apply the previous result to cliques in a min-max clique covering
with simple intersection.

\begin{lem}\label{lem:onlyone}
  Assume that $G$ is a graph and $\{C_1,C_2, \dots, C_\ell\}$ is a
  min-max clique covering with simple intersection.  If $S$ is a positive zero
  forcing set, then the sets $(V(G) \backslash S) \cap C_{i,j}$ and
  $(V(G) \backslash S) \cap C_{i,i}$ for any $i,j \in \{1,\dots,n\}$
  have size at most one.
\end{lem}
\begin{proof}

  Assume that there are two vertices $u_1,u_2$ in $(V(G) \backslash S)
  \cap C_{i,j}$ (where $i \neq j$). Then in the positive zero forcing process
  in which the vertices in $S$ are initially black, both $u_1$ and
  $u_2$ are initially white.

  At some point in the positive zero forcing process there is a vertex $v$
  which forces $u_1$ (assuming that $u_2$ is not forced before
  $u_1$). Since $u_1$ and $u_2$ are adjacent, they must be in the same
  component. Since $v$ is adjacent to $u_1$, it must be in either
  $C_i$ or $C_j$ and thus it is also adjacent to $u_2$. But since
  $u_2$ is white at this stage $v$ cannot force $u_1$.

  A similar argument shows that there cannot be two vertices in $(V(G)
  \backslash S) \cap C_{i,i}$.
\end{proof}

\begin{thm}\label{zplus-compress}
  Let $G$ be a graph (that is connected, but not a clique) in which the
  maximal cliques have simple intersection and let $\cg{G}$ be the
  compressed cliques graph of $G$.  Then
\[
|V(G)| - Z_+(G) = |V(\cg{G})| - Z_+(\cg{G})
\]
and there exist forcing trees for $G$ and $\cg{G}$ that differ only in
that these forcing trees for $G$ are isolated vertices.
\end{thm}

\begin{proof}
  Let $S'$ be an optimal positive zero forcing set for $\cg{G}$; then
  the vertices in set $V(\cg{G}) \backslash S'$ are initially coloured
  white in the positive zero forcing process.  Using this set of
  vertices, we can construct a positive zero forcing set for $G$ as
  follows. If $v_{i,i}$ is in the set $V(\cg{G}) \backslash S'$, then
  colour exactly one of the vertices in $C_{i,i}$ white. Similarly, if
  $v_{i,j}$ is in the set $V(\cg{G}) \backslash S$, then colour
  exactly one of the vertices in $C_{i,j}$ white. Colour all the
  vertices that are not white, black. We claim that the set of black
  vertices is a positive zero forcing set in $G$. To see this,
  consider that during the positive zero forcing process on $\cg{G}$ the
  vertex $v_{i,j}$ forces the vertex $u_{a,b}$. In this case every vertex in
  $C_{i,j}$ (in $G$) is black and any one of these vertices can force
  the single white vertex in $C_{a,b}$.  This also shows that the set
  of forcing trees for $G$ and $\cg{G}$ only differ in that the set of
  forcing trees for $G$ may have more isolated vertices.

  Thus the number of vertices initially coloured white in $G$ is at
  least as the number initially colour white in $\cg{G}$; that is
\[
|V(G)| - Z_+(G) \geq |V(\cg{G})| - Z_+(\cg{G}).
\]

Conversely, if $S$ is a positive zero forcing set for $G$, then the vertices the $V(G)
\backslash S$ are all initially coloured white. Using these vertices
it is possible to construct a positive zero forcing set for $\cg{G}$. From
Lemma~\ref{lem:onlyone}, there is at most one vertex in $C_{i,j}$ for
any pair $i,j\in \{1, \dots, \ell\}$ that is initially coloured white.
If one vertex in $C_{i,j}$ is white, then the vertex $v_{i,j}$  (the representative chosen
in the construction of $\cg{G}$) in $\cg{G}$
is initially coloured white; all other vertices are coloured black.

Assume that at a step in the positive zero forcing process on $G$ the vertex
$v$ forces $u$. Then $v$ must be black at this step, and both $v$ and $u$ must
both belong to some maximal clique, say $C_i$. Consider the following
four cases.

{\bf Case 1:} Assume $u \in C_{i,i}$ and $v \in C_{i,j}$ for some
$j$. Since $v$ forces $u$, all the vertices in $C_i\backslash u$ must
already be black (since they are all adjacent to $v$). Any vertex
$v_{i,a}$ (except $v_{i,i}$ which is white, since $u$ is white) must
be black; since these are the only vertices adjacent to $v_{i,j}$, the
vertex $v_{i,j}$ can force $v_{i,i}$ in $\cg{G}$.

{\bf Case 2:} Assume that $u \in C_{i,i}$ and $v \in C_{i,i}$. Then all the vertices in
$C_i$, except $u$, must be black. When the black vertices are removed,
$u$ is an isolated vertex, so it can also be forced by any vertex in
$C_{i,j}$. This is Case 1.

{\bf Case 3:} Assume $u \in C_{i,j}$ and $v \in C_{i,i}$ for some $j$. Since $v$ can
force $u$, all the vertices in $C_i\backslash u$ must be black. Thus
every vertex $v_{i,a}$, except $v_{i,j}$, must be black in $\cg{G}$. The only
white neighbour of $v_{i,i}$ is $v_{i,j}$. Hence $v_{i,i}$ can force $v_{i,j}$ in $\cg{G}$.

{\bf Case 4:} Assume $u \in C_{i,j}$ and $v \in C_{i,k}$ for $j\neq
k$. Since $v$ can force $u$, the only white vertex adjacent to $v$ in
the same component as $u$ is $u$. This implies that all other vertices in $C_{i,k}$ are black.
Thus by Proposition~\ref{prop:phiconnects}, $v_{i,j}$ is the only white
vertex adjacent to $v_{i,k}$ in its components, so $v_{i,k}$ can force $v_{i,j}$.

{\bf Case 5:} If $u \in C_{i,j}$ and $v \in C_{i,j}$ for some $j$, then all the
vertices in $C_i$ and $C_j$, except $u$, are black. If the black
vertices are removed, then $u$ will be isolated, so any vertex in $(C_i
\cup C_j) \backslash C_{i,j}$ can force $u$ (this set cannot be empty
since $C_i$ and $C_j$ are distinct, maximal cliques). Then this is either Case 3 or Case 4.
\end{proof}

\section{Connections with Johnson graphs}
\label{sec:Jonhson}

The compressed cliques graph is related to the \emph{Johnson graph}
$J(m,2)$ \cite{HGT}.  The graph $J(m,2)$ has the set of pairs from $\{1,2,\dots,m\}$ as
its vertex set and two vertices are adjacent if and only if they have
non-empty intersection.  This graph is the line graph of the
complete graph.

We will use a generalization of this graph that we denote by
$J'(m,2)$. The vertices of $J'(m,2)$ are the set $\{i,j\}$ and $\{i\}$
where $i,j \in \{1,2,\dots,m\}$ and two sets are adjacent if and only
if they have non-empty intersection. These graphs play a major role in
the theory of compressed cliques graphs.

\begin{lem}
  If $G$ is a graph with $\CC(G)>1$ and $G$ has a min-max clique
  covering that satisfies simple intersection, then the compressed
  cliques graph of $G$ will be an induced subgraph of $J'(\CC(G),2)$.
\end{lem}

There is a simple clique covering for this generalization of the
Johnson graph. Define $C_i$ to be the set of all vertices in $J'(m,2)$
that contain $i$. Then $C_i$ is a maximal clique and $\mathcal{C} =
\{C_1,\dots, C_m\}$ is a min-max clique covering that has simple
intersection.

\begin{lem}
For any $m>3$,
\begin{enumerate}
\item $|V(J'(m,2))| = \binom{m}{2} + m$,
\item $\CC(J'(m,2)) = m$,
\item $Z_+(J'(m,2)) =  \binom{m}{2}$,
\item $Z_+(J'(m,2))  = |V(J'(m,2))| - \CC(J'(m,2))$.
\end{enumerate}
\end{lem}
\begin{proof}
  The first statement is clear.

  For each $i$ define a clique $C_i$ that consists of all the vertices
  that contain $i$; clearly $\{ C_1,C_2, \dots, C_m\}$ is a clique
  covering of $J(m,2)'$ of size $m$. Further, each of the edges
  $\{\{i\}, \{i,1\}\}$ for $i \in \{1, 2, \dots, m\}$ must be in a
  distinct clique, so the clique cover number cannot be any smaller
  than $m$.

  The set of all vertices of the form $(i,j)$ with $i \neq j$ is a
  positive zero forcing set---if these vertices are removed, then the
  remaining graph is an independent set. Thus this positive zero
  forcing set is optimal from the well known bound
$|V(G)| - \CC(G) \leq Z_+(G)$ 
(this is Lemma~\ref{lem:ccbound}).
\end{proof}

For the graph $J'(m,2)$, in an optimal positive zero forcing process
there is initially one white vertex for each clique in the min-max clique
covering. In fact, we can assign each vertex $\{i\}$ to be white, then
within each maximal clique there is exactly one white vertex.

Similarly, the Johnson graph $J(m,2)$ also has a min-max clique covering with
simple intersection (using the same definition for the sets $C_i$).

\begin{lem}
For any $m>3$,
\begin{enumerate}
\item $V(J(m,2)) = \binom{m}{2}$,
\item $\CC(J(m,2)) = m$,
\item $Z_+(J(m,2)) =  \binom{m}{2}-m+2$,
\item $Z_+(J(m,2))  = |V(J(m,2))| - \CC(J(m,2)) + 2$.
\end{enumerate}
\end{lem}

In this case we need that $m>3$ since the vertices of graph $J(3,2)$
are $\{1,2\}, \{1,3\}, \{2,3\}$ and these form a complete graph on
three vertices.  This is an exceptional maximal clique in $J(m,2)$;
every other maximal clique consists of all sets that contain a fixed element.
In fact, any maximal clique in the compressed cliques graph with more than three
vertices will be the set of all vertices whose corresponding sets all
contain a common element. 


\section{Forbidden subgraphs of compressed cliques graphs}

In the previous section, we observed that the compressed cliques graph is
an induced subgraph of the graph $J'(m,2)$, when $m>3$. In this section we illustrate some
forbidden subgraphs of the compressed cliques graph.

A \emph{claw} is a graph with four vertices and three edges in which one
vertex is adjacent to the other three (this is the complete bipartite
graph $K_{1,3}$). The first graph in Figure~\ref{fig:forbidden} is a
claw. A graph is called \emph{claw-free} if no set of four vertices
induce a subgraph that is a claw. Claw-free graphs have been widely
studied, see, for example, the survey \cite{FFR}.

\begin{prop}
  If $G$ is a graph that has a min-max clique covering with simple
  intersection, then both $G$ and the compressed cliques graph
  $\cg{G}$ are claw-free.
\end{prop}
\begin{proof}
  Let $\mathcal{A}$ be a min-max clique covering of $G$ that satisfies
  simple intersection.  If $G$ has a claw, then the three edges in the
  claw belong to three different cliques. Thus the center vertex in
  the claw belongs to three different cliques. 

Assume that the compressed cliques graph $\cg{G}$ has a claw and that
$v_{i,j}$ is the vertex of degree three in the claw. Then each of the
other three vertices in the claw must adjacent to $v_{i,j}$, so each
will correspond to a set that contains one of $i$ or $j$. This then
implies that at least two of the these three vertices will both
contain $i$ or will both contain $j$ and thus be adjacent, which is a
contradiction.
\end{proof}

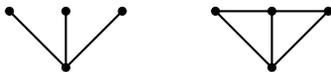
\begin{figure*}[h]
\begin{center}
\begin{tikzpicture}[scale=.75]
\draw [fill] (1,0)  circle (2pt) node [above] {};
\draw [fill] (0,1)  circle (2pt) node [above] {};
\draw [fill] (1,1)  circle (2pt) node [above] {};
\draw [fill] (2,1)  circle (2pt) node[above] {};
\draw [thick] (1, 0) -- (0,1);
\draw [thick] (1, 0) -- (1,1);
\draw [thick] (1, 0) -- (2,1);
\end{tikzpicture}
\qquad
\begin{tikzpicture}[scale=.75]
\draw [fill] (1,0)  circle (2pt) node [above] {};
\draw [fill] (0,1)  circle (2pt) node [above] {};
\draw [fill] (1,1)  circle (2pt) node [above] {};
\draw [fill] (2,1)  circle (2pt) node[above] {};
\draw [thick] (1, 0) -- (0,1) -- (1,1);
\draw [thick] (1, 0) -- (1,1) -- (2,1);
\draw [thick] (1, 0) -- (2,1);
\end{tikzpicture}
\end{center}
\caption{Forbidden induced subgraphs for compressed cliques graphs.\label{fig:forbidden}}
\end{figure*}

Note the graph on the right in Figure~\ref{fig:forbidden} is special
in the sense that it alone cannot be a compressed cliques graph, but
rather if a compressed cliques graph contains the subgraph on the right
as an induced subgraph, then it must contain other vertices and edges.  For
example, consider the graph $T_3$ given in Figure~\ref{t3}.

\begin{figure*}[h]
\begin{center}
\begin{tikzpicture}
\GraphInit[vstyle=simple]
\tikzset{VertexStyle/.style = {shape = circle,fill = black,minimum size = 5pt,inner sep=0pt}}
\Vertex[x=0, y=0] {E} \Vertex[x=-1, y=0] {D} \Vertex[x=1, y=0] {F}
\Vertex[x=-0.5, y=1] {B} \Vertex[x=0.5, y=1] {C} 
\Vertex[x=0, y=2] {A}

\Edge(A)(B) \Edge(A)(C)  
\Edge(B)(C) 
\Edge(B)(D) \Edge(B)(E) 
\Edge(C)(E) \Edge(C)(F)
\Edge(D)(E)  \Edge(E)(F) 
\end{tikzpicture}
\end{center}
\caption{$T_3$: an example of a self-compressed graph.}\label{t3}
\end{figure*}
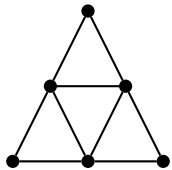

Observe that the graph $T_3$ has clique cover number three, and this
min-max clique cover has simple intersection. Furthermore, $\cg{T_3}$
is itself $T_3$, and includes the graph on the right in
Figure~\ref{fig:forbidden} as an induced subgraph.

Before we come to our next claim regarding compressed cliques graphs,
we define some particular subgraphs. We call a cycle $u_1u_2 \cdots u_tu_1$
\emph{suspended} when exactly one of $u_1, u_2, \ldots, u_t$ has
degree larger than two in $G$, and all of the remaining vertices from
this set have degree two in $G$.

A graph $G$ contains an \emph{ear}, if it contains $K_3$ as a suspended
cycle. Observe that if $G$ contains a suspended cycle of length
at least four, and it is not just a cycle itself, then $G$ must contain a claw, and
hence is forbidden as a compressed cliques graph. 

\begin{lem}
  If $G$ is a graph with a min-max clique cover satisfying simple
  intersection, then $\cg{G}$ does not contain a suspended cycle.
\end{lem}

\begin{proof}
  Suppose the compressed cliques graph contains a suspended cycle,
  from the comments above, the suspended cycle must be an ear. Assume
  the vertices of the suspended cycle are $u$ and $w$, both of degree
  two, and $x$ with degree at least three. Then, it follows that at
  least one of $u$ or $w$ must have a label of distinct indices
  $\{i,j\}$, namely, $u=v_{i,j}$. In this case, if the label of $w$ is
  $v_{i,i}$ (or $v_{j,j}$), then the label for $x$ must contain an $i$
  (since it is adjacent to $w$), but also $j$, since other $u$ would
  be adjacent to another vertex outside of the suspended cycle.
  This means that both $u$ and $x$ have the same label in $\cg{G}$.

  If the label of $w$ is given by $w=v_{i,k}$ (or $v_{j,k}$), where
  $k$ is different from $i$ or $j$, then since $u$ and $w$ have degree
  two, the label on $x$ must be $\{k,j\}$. Then any other vertex
  which is adjacent to $x$, must also be adjacent to one of $u$ or
  $w$.
\end{proof}

\section{Vertex-clique graphs}

For any graph $G$ with vertex set $V$ and edge set $E$, we
construct a new graph $H$, obtained from $G$, that has a min-max covering with
simple intersection. For each vertex $v \in V$, we
construct a clique $K_{d(v)}$ where $d(v)$ is the degree of $v$ in
$G$. For each edge $\{u, v\} \in E$, we add an edge between a vertex
in $K_{d(u)}$ and a vertex in $K_{d(v)}$ such that each vertex in
$K_{d(u)}$ (or $K_{d(v)}$) has at most one neighbour outside of
$K_{d(u)}$ (or $K_{d(v)}$).  This new graph is called the
\emph{vertex-clique graph of $G$}. By definition a vertex-clique graph
cannot be a complete graph. The next lemma provides some interesting
properties of a vertex-clique graph. Recall that a \emph{line graph} 
of a given graph $X$, denoted by  $L(X)$, is obtained by associating
a vertex with each edge of $X$ and connecting two vertices with 
an edge if and only if the corresponding edges of $X$ have a vertex in common.

 \begin{lem}\label{lem:vertex-clique-graph} Let $H$ be a vertex-clique graph of $G$. Then
\begin{enumerate}
\item
$H$ is a line graph;
\item
$H$ has a min-max clique covering that satisfies simple intersection; and
\item
$\mathcal{C}(H)=H$.
\end{enumerate}
\end{lem}
\begin{proof}
  If $H$ is the vertex-clique graph of $G$, then $H$ is the line graph
  of the graph formed by performing an edge subdivision on each edge
  of $G$.

  Let $C_v$ be the clique in $H$ that replaces the vertex $v$ in $G$,
  similarly, let $C_{\{u,v\}}$ be the edge in $H$ between a vertex in
  $C_v$ and $C_u$. The set of all $C_v$ and $C_{\{u,v\}}$ form a
  min-max clique covering that satisfies simple intersection.

  Finally we show that $\cg{H} = H$. Note that the set $C_v \cap
  C_u$ is empty, and $C_v \cap C_{\{u,v\}} = \{v\}$. Consider the set
  $C_{\{u,v\}} \cap C_{\{u',v'\}}$. If $C_{\{u,v\}}$ is an edge in
  $H$, then there is no edge $C_{\{u',v\}}$, so any $C_{\{u,v\}}$
  and $C_{\{u',v'\}}$ will be disjoint. Thus the
  intersection of any two cliques in this covering is either empty, or
  has size one. By Theorem~\ref{thm:C(G)isG}, this implies that
  $\mathcal{C}(H)$ is isomorphic to $H$.
\end{proof}

For a graph $G$, we can form a simplified graph called the
\emph{reduced graph of $G$} that has the same positive zero forcing number
as the original graph. First we define an induced path $u_1u_2 \cdots u_t$ in a graph to
be a \emph{suspended path} if the vertices $u_2, u_3, \dots, u_{t-1}$
all have degree two. To form the reduced graph of $G$, first
recursively delete all vertices with degree one.  Once all the
vertices of degree one have been removed, contract any induced
suspended paths of length at least two to an edge.  The graph that
remains after performing both of these operations is called the \emph{reduced graph of
$G$}, and is denoted by $\mathcal{R}(G)$. It is clear that these two
operations do not effect the size of a positive zero forcing set. 
\begin{lem}\label{lem:reduced-graph}
Let $G$ be a graph, then $Z_+(G)=Z_+(\mathcal{R}(G))$.
\end{lem}

Applying this reduction to a vertex-clique graph produces a new graph for
which we can determine bounds on the positive zero forcing number.

\begin{theorem}\label{thm:forest}
Let $G$ be a connected vertex-clique graph and $\mathcal{R}(G)$ be the reduced
graph of $G$.
\begin{enumerate}
\item If $\mathcal{R}(G)$ has only one vertex, then $Z_+(G)=1$.
\item If $\mathcal{R}(G)$ is a 3-cycle, then $Z_+(G)=2$.
\item If $\mathcal{R}(G)$ has more than three vertices, let $k$ be the
  number of edges in $\mathcal{R}(G)$ that are themselves maximal
  cliques in $\mathcal{R}(G)$, then $Z_+(G) \leq k$. Further, this upper
  bound can be tight for some graphs $G$.
\end{enumerate}
\end{theorem}

\begin{proof}
  Note that each cycle in $G$ is either unchanged in
  $\mathcal{R}(G)$ or reduced to a cycle of smaller length.  If
  $\mathcal{R}(G)$ has only one vertex, then $G$ is a path, and
  $Z_+(G)=1$.  If $\mathcal{R}(G)$ is a 3-cycle, then $G$ is a
  unicyclic graph, and $Z_+(G)=2$.

  Suppose that $\mathcal{R}(G)$ has more than three vertices.  Let
  $C^{(2)}$ be the set of all maximal cliques of size two in
  $\mathcal{R}(G)$, and let $C^{(\geq 3)}$ be the set of all maximal
  cliques of size at least three in $\mathcal{R}(G)$. Since $G$ is a
  connected vertex-clique graph, we know that every vertex in
  $\mathcal{R}(G)$ must belong to exactly two maximal cliques, one in
  $C^{(2)}$ and the other in $C^{(\geq 3)}$.  Note that $C^{(\geq 3)} \not=
  \emptyset$ as $\mathcal{R}(G)$ has more than three vertices.
  Color all vertices of $\mathcal{R}(G)$ using the following
  procedure.

\begin{enumerate}
\item Pick a clique $C \in C^{(\geq 3)}$, color all vertices of $C$
  black.  For each edge in $C^{(2)}$, if one endpoint is black and the
  other is not coloured, then color the uncoloured vertex to grey.

\item \label{alg:s2}
If there is a clique $C \in C^{(\geq 3)}$ that contains both
grey and uncoloured vertices, then color all uncoloured vertices of $C$ to black.

\item
For each edge in $C^{(2)}$, if one endpoint is black and the other is uncoloured,
then color the uncoloured vertex to grey.

\item
If all vertices of $\mathcal{R}(G)$ are coloured, then stop;
otherwise, go to Step~\ref{alg:s2}.
\end{enumerate}

In the above procedure, when a vertex $u$ is coloured black, there must
be a neighbour $v$ such that $\{u, v\} \in C^{(2)}$ and $v$ is coloured grey
because every vertex in $\mathcal{R}(G)$ belongs to exactly two
maximal cliques (one in $C^{(2)}$ and the other in $C^{(\geq 3)}$).  Thus,
every black vertex in $\mathcal{R}(G)$ is an endpoint of an edge in
$C^{(2)}$. On the other hand, for each edge in $C^{(2)}$, at most one endpoint
is coloured black. As all black vertices form a positive zero
forcing set of $\mathcal{R}(G)$, it follows from
Lemma~\ref{lem:reduced-graph} that $Z_+(G) = Z_+(\mathcal{R}(G)) \leq
|C^{(2)}|$.
\end{proof}

Finally, we give an example that shows the upper bound in the
theorem above cannot be improved in general.  Let $H$ be the
vertex-clique graph of complete bipartite graph $K_{2,3}$ (see Figure \ref{fig:reduce}). Then
$\mathcal{R}(H) = K_2 \Box K_3$, where $\Box$ denotes the Cartesian product
of graphs. Note that the number of edges in $K_2
\Box K_3$ that are themselves maximal cliques is three, and
$Z_+(H)=3$. In general, if $H$ is the vertex clique graph of $K_{2,n}$, then $\mathcal{R}(H) = K_2
\Box K_n$. In this case, $Z_+(\mathcal{R}(H)) = n$ and the number of
edges in $\mathcal{R}(H) $ that form maximal cliques is precisely $n$.

\begin{figure*}[h]
\begin{center}
\begin{tabular}{ccc}
\begin{tikzpicture}
\GraphInit[vstyle=simple]
\tikzset{VertexStyle/.style = {shape = circle,fill = black,minimum size = 5pt,inner sep=0pt}}
\Vertex[x=0, y=1] {A}   
\Vertex[x=-1, y=0] {B}   \Vertex[x=0, y=0] {C} \Vertex[x=1, y=0]{D}
\Vertex[x=0, y=-1] {E}
\Edge(A)(B) \Edge(A)(C) \Edge(A)(D)
\Edge(E)(B) \Edge(E)(C) \Edge(E)(D)
\end{tikzpicture}
&
\begin{tikzpicture}
\GraphInit[vstyle=simple]
\tikzset{VertexStyle/.style = {shape = circle,fill = black,minimum size = 5pt,inner sep=0pt}}
\Vertex[x=-.25, y=1.25] {A1}   \Vertex[x=0, y=.75] {A2}   \Vertex[x=.25, y=1.25] {A3}   
\Vertex[x=-1, y=-.25] {B1}   \Vertex[x=0, y=-.25] {C1} \Vertex[x=1, y=-.25]{D1}
\Vertex[x=-1, y=.25] {B2}   \Vertex[x=0, y=.25] {C2} \Vertex[x=1, y=.25]{D2}
\Vertex[x=-.25, y=-1.25] {E1}\Vertex[x=0, y=-.72] {E2}\Vertex[x=.25, y=-1.25] {E3}
\Edge(A1)(A2) \Edge(A1)(A3) \Edge(A3)(A2)
\Edge(E1)(E2) \Edge(E1)(E3) \Edge(E3)(E2)
\Edge(A1)(B2) \Edge(A2)(C2) \Edge(A3)(D2)
\Edge(B2)(B1) \Edge(C2)(C1) \Edge(D2)(D1)
\Edge(B1)(E1) \Edge(C1)(E2) \Edge(D1)(E3)
\end{tikzpicture}
&
\begin{tikzpicture}
\GraphInit[vstyle=simple]
\tikzset{VertexStyle/.style = {shape = circle,fill = black,minimum size = 5pt,inner sep=0pt}}
\Vertex[x=-.5, y=1.25] {A1}   \Vertex[x=0, y=.75] {A2}   \Vertex[x=.5, y=1.25] {A3}   
\Vertex[x=-.5, y=-1.25] {E1}\Vertex[x=0, y=-.72] {E2}\Vertex[x=.5, y=-1.25] {E3}
\Edge(A1)(A2) \Edge(A1)(A3) \Edge(A3)(A2)
\Edge(E1)(E2) \Edge(E1)(E3) \Edge(E3)(E2)
\Edge(A1)(E1) \Edge(A2)(E2) \Edge(A3)(E3)
\end{tikzpicture} \\
$K_{2,3}$ & $H$ (the vertex-clique graph of $K_{2,3}$) & $\mathcal{R}(H)$ \\
\end{tabular}
\caption{Example of a graph where the bound in Theorem~\ref{thm:forest} is tight.}\label{fig:reduce}
\end{center}
\end{figure*}
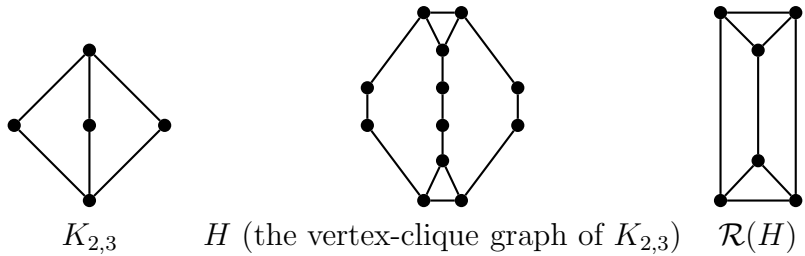

\section{Examples of compressed cliques graphs}

We now consider several examples of compressed cliques graphs $G$ with an eye
on comparing the values of $Z_+(G)$ and $|V(G)| -\CC(G)$.
Our focus on these graphs comes from the following well known lemma (see also \cite{zplus}).

\begin{lem}\label{lem:ccbound}
For any graph $G$ 
\[
|V(G)| - \CC(G) \leq Z_+(G).
\]
\end{lem}

Our first result gives a large family for graphs for which the
positive zero forcing number is equal to the difference between the
number of vertices and the clique covering number.  If $\cg{G}$
contains no \emph{induced} cycles (other than $K_3$), then $G$ is a
chordal graph in which a min-max clique covering satisfies simple
intersection. Theorem 4.6.5 of~\cite{FatemehThesis} states that for a
chordal graph $G$, $Z_+(G) = |V(G)|-\CC(G)$. Thus we have the
following result.

\begin{theorem}
If $G$ has a min-max clique cover with simple intersection and the compressed cliques
graph has no induced cycles, other than $K_3$, then
\[
Z_+(G) = |V(G)|-\CC(G).
\]
\end{theorem}

If $G$ consists of a series of cliques $C_1, \dots, C_k$ in which only
consecutive cliques intersect, then we call $G$ a \emph{path of
  cliques}. In this case we have equality in the inequality of
Lemma~\ref{lem:ccbound}.

\begin{cor}
Let $G$ be a graph that is a path of cliques, then
\[
|V(G)| - \CC(G) = Z_+(G).
\]
\end{cor}

The next family of graphs that we consider are a generalization of the {\em
  musical graph} $M_n$ defined in \cite{Knuth}. For $n\geq3$, the
graph $M_n$ has $2n$ vertices, $5n$ edges, and is isomorphic to the
Cayley graph 
$Cay( \mathbb{Z}_{2n}, \{ \pm 1, \pm (n-1), n\} )$. Figure~\ref{M5} is the musical graph $M_4$.

\begin{figure*}[h]
\begin{center}
\begin{tikzpicture}[scale=.75]
\GraphInit[vstyle=simple]
\tikzset{VertexStyle/.style = {shape = circle,fill = black,minimum size = 5pt,inner sep=0pt}}
\Vertex[x=0, y=0] {A} \Vertex[x=0, y=3] {B}
\Vertex[x=0.5, y=0.5] {C}\Vertex[x=0.5, y=2.5] {D}
\Vertex[x=2.5, y=0.5] {E} \Vertex[x=2.5, y=2.5] {F}
\Vertex[x=3, y=0] {G} \Vertex[x=3, y=3] {H}

\Edge(A)(B) \Edge(A)(C) \Edge(A)(D) \Edge(A)(E) \Edge(A)(G)
\Edge(B)(C) \Edge(B)(D) \Edge(B)(F) \Edge(B)(H)
\Edge(C)(D) \Edge(C)(E) \Edge(C)(G)
\Edge(D)(F) \Edge(D)(H)
\Edge(E)(F) \Edge(E)(G) \Edge(E)(H)
\Edge(F)(G) \Edge(F)(H) 
\Edge(G)(H)

\end{tikzpicture}
\end{center}
\caption{The musical graph $M_4$.} \label{M5} 
\end{figure*}
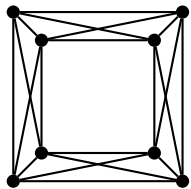

The graph $M_n$ has a min-max clique covering with $n$ cliques that
satisfies simple intersection. The positive zero forcing number for
$M_n$ can then be easily calculated.

\begin{lem}
Let $n \geq 3$, then $Z_{+}(M_n)=n+2$.
\end{lem}
\begin{proof}
Observe that the compressed cliques graph of musical graph $M_n$ is $C_n$. Therefore,
using Theorem~\ref{zplus-compress} we have
\[
|V(M_n)|-Z_{+}(M_n)=|V(\mathcal{C}(M_{n}))|-Z_{+}(\mathcal{C}(M_{n}))=n-2
\]
and so $Z_{+}(M_n)=n+2$. 
\end{proof}

The musical graph is an example of a graph for which
\[
Z_+(G) \leq |V(G)| -|\CC(G)| +2.
\]
We can generalize this to a family of related graphs.
Let $G$ be a graph for which $\{C_1,C_2,\dots,C_\ell\}$ is a min-max
clique covering.  If each of $C_i \cap C_{i+1}$ for $i = 1, \dots, \ell
- 1$, as well as $C_1 \cap C_\ell$, are non-empty (these are the
\emph{consecutive} cliques), while all other intersections of these
cliques are empty, then $G$ is called a \emph{cycle of cliques}.
These graphs are a generalization of the graphs defined as
$C_t(K_n)$~\cite{PetersThesis}.

\begin{thm}
Let $G$ be a cycle of cliques, then
\[
Z_+(G) \leq |V(G)| - \CC(G) +2.
\]
\end{thm}
\begin{proof}
Assume that $G$ is a cycle of cliques $\{C_1,C_2,\dots,C_\ell\}$.
For $i = 2,3, \dots, \ell-1$, let $v_i$ be a vertex in $C_{i} \cap
C_{i+1}$ (since $G$ is a cycle of cliques, this is possible). Colour
the vertices $v_2, \dots, v_{\ell-1}$ white and all other vertices in
$G$ black. We claim that this is a positive zero forcing set for $G$.
To see this, note that any vertex in $C_1 \cap C_2$ can force
$v_{2,3}$, which in turn forces $v_{3,4}$. Continuing in this manner, all
the vertices in $G$ will be forced to black.
\end{proof}

There is a subfamily of the cycles of cliques for which the positive
zero forcing number is equal to the number of vertices minus the
clique covering number of the graph.

\begin{thm}
  Let $G$ be a cycle of cliques labeled $\{C_1,C_2, \dots, C_\ell\}$.
  If there are two cliques $C_i$ and $C_j$ with $C_{i,i}$ and
  $C_{j,j}$ non-empty, then
\[
Z_+(G) = |V(G)| - \CC(G).
\]
\end{thm}
\begin{proof}
  Assume $i$ is the smallest $i$ such that $C_{i,i}$ is non-empty and
  $j$ is the largest value for which $C_{j,j}$ is non-empty. 

  The clique compressed cliques graph for $G$ will contain a cycle with vertices
  labeled $v_{i,i+1}$ for $i = 1, \dots, \ell-1$ and
  $v_{1,\ell}$. Since the sets $C_{i,i}$ and $C_{j,j}$ are non-empty
  the graph will also have vertices $v_{i,i}$ (adjacent to $v_{i-1,i}$
  and $v_{i,i+1}$), and $v_{j,j}$ (adjacent to $v_{j-1,j}$ and $_{j,j+1}$).

  Colour the vertices $v_{i,i}$ and $v_{j,j}$ white along with every
  vertex $v_{k,k+1}$ in the cycle, except the vertices in the clique
  $C_j$. The remaining vertices are coloured black. In this colouring,
  there are $\ell$ white vertices, so the number of black vertices is
  $|V(G)| - \ell$. We show that the black vertices form a
  positive zero forcing set for $G$.

The vertex $v_{j,j+1}$ will force $v_{j,j}$ and then start a forcing
sequence by forcing $v_{j+1, j+2}$ and continuing to $v_{i-1,i}$.
Similarly, $v_{j-1,j}$ will force the vertex $v_{j-2,j-1}$, and
continue forcing along the cycle until the vertex $v_{i,i+1}$. Finally,
$v_{i,i+1}$ can force $v_{i,i}$.
\end{proof}


Our next example is a family of graphs for which there is a large gap
between the positive zero forcing number and the number of vertices minus the
clique covering number of the graph.

Define a graph $X(n;\ell_1,\dots,\ell_k)$ that has vertices $x_1, \dots ,x_n$ which induce a
clique. In addition, this graph also contains disjoint cycles $C_1, \ldots, C_k$,
in which $C_i$ has length $\ell_i$, and each cycle contains
exactly two vertices from the set $\{x_1 \dots,x_n\}$ (see also Figure \ref{xgraph} where
$X(8;4,4,4,4)$ is depicted). In this case, the number of vertices in
this graph is $n + \sum_{i=1}^k (\ell_i -2)$ and $\CC(X) = 1 +
\sum_{i=1}^k (\ell_i -1)$. Moreover, the only such clique cover is a min-max clique cover
that has simple intersection.

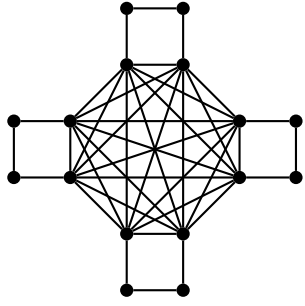
\begin{figure*}[h]
\begin{center}
\begin{tikzpicture}[scale=.75]
\GraphInit[vstyle=simple]
\tikzset{VertexStyle/.style = {shape = circle,fill = black,minimum size = 5pt,inner sep=0pt}}
\Vertex[x=0, y=1] {A} \Vertex[x=-1, y=1] {A1}
\Vertex[x=0, y=2] {B}\Vertex[x=-1, y=2] {B1}

\Vertex[x=1, y=0] {C}\Vertex[x=1, y=-1] {C1}
\Vertex[x=1, y=3] {D}\Vertex[x=1, y=4] {D1}

\Vertex[x=2, y=0] {E}\Vertex[x=2, y=-1] {E1}
\Vertex[x=2, y=3] {F}\Vertex[x=2, y=4] {F1}

\Vertex[x=3, y=1] {G} \Vertex[x=4, y=1] {G1}
\Vertex[x=3, y=2] {H} \Vertex[x=4, y=2] {H1}

\Edge(A)(B) \Edge(A)(C) \Edge(A)(D) \Edge(A)(E) \Edge(A)(F) \Edge(A)(G) \Edge(A)(H)
\Edge(B)(C) \Edge(B)(D) \Edge(B)(E) \Edge(B)(F) \Edge(B)(G) \Edge(B)(H)
\Edge(C)(D) \Edge(C)(E) \Edge(C)(F) \Edge(C)(G) \Edge(C)(H)
\Edge(D)(E) \Edge(D)(F) \Edge(D)(G) \Edge(D)(H)
\Edge(E)(F) \Edge(E)(G) \Edge(E)(H)
\Edge(F)(G) \Edge(F)(H) \Edge(G)(H)

\Edge(A1)(B1) \Edge(D1)(F1)  \Edge(E1)(C1) \Edge(G1)(H1)
\Edge(A1)(A) \Edge(C1)(C) \Edge(E1)(E) \Edge(G1)(G)
\Edge(B)(B1) \Edge(D)(D1) \Edge(F)(F1) \Edge(H)(H1)

\end{tikzpicture}
\end{center}
\caption{The graph $X(8;4,4,4,4)$.}\label{xgraph}
\end{figure*}

\begin{thm}
  Let $n$, $\ell_1, \dots, \ell_k$ be integers with $n \geq \sum_i
  \ell_i$.  Then $X=X(n;\ell_1,\dots,\ell_k)$ has a min-max clique covering with
  simple intersection such that
  \[Z_{+}(X(n;\ell_1,\dots,\ell_k)) =
n-1 = |V(X)| -\CC(X) + k.
\]
\end{thm}
\begin{proof}
To verify this, colour all but one of the vertices in
the $n$-clique black.
\end{proof}

\section{Further work}

We have shown that for graphs that have a clique covering satisfying
simple intersection, we can determine the positive zero forcing number
of the graph from the positive zero forcing number of the compressed
cliques graph. Further, any compressed cliques graph will be an induced
subgraph of the generalization of the Johnson graph $J'(n,2)$ given in
Section~\ref{sec:Jonhson}. This graph plays a special role in this
theory.

In this paper, we defined a simple intersecting clique covering of a
graph. This means that any vertex of the graph is contained in at most
two cliques in the covering. Clearly, this can be generalized. We will
say a clique covering of a graph has \emph{$s$-wise intersection} if
any vertex is contains in at most $s$ cliques in the clique
covering. Then we could generalize the Johnson graph $J(n,s)$, to the
graph $J'(n,s)$. The vertices of this graph will be subsets of size at
most $s$ from $\{1, \dots, n\}$ and two vertices will be adjacent if
and only if there sets are intersecting. Then every graph will have a
clique covering that has $s$-wise intersection for $s$ sufficiently
large. Then for any graph $G$, the compressed cliques graph of $G$ is
an induced subgraph of $J'(n,s)$. Moreover, if we can determine the positive
zero forcing number for the compressed cliques graph, then we can determine
the positive zero forcing number for the original graph $G$.

\end{document}